\documentclass[11pt,a4paper,twoside]{amsart}
\usepackage{amsmath,amssymb,amsfonts,amsthm}
\usepackage{hyperref}
\usepackage{lineno}

\def\k{\kappa}

\def\l{\lambda}
\def\to{\longrightarrow}
\def\mto{\longmapsto}
\def\a{\alpha}

\def\g{\gamma}
\def\p{\phi}

\def\G{\Gamma}
\def\s{\psi}

\def\eps{\epsilon}
\def\r{\rho}
\def\o{\circ}
\def\t{\tau}
\def\vp{\varphi}

\def \dar{\times}
\def \vlim{\varprojlim}

\def\N{\mathbb{N}}

\def\F{\mathbb{F}}
\def\R{\mathbb{R}}
\def\E{\mathbb{E}}
\def\I{\mathbb{I}}

\def\pa{\partial}
\newtheorem{The}{Theorem}[section]
\newtheorem{Pro}[The]{Proposition}
\newtheorem{Lem}[The]{Lemma}
\newtheorem{Cor}[The]{Corollary}
\theoremstyle{definition}
\newtheorem{Def}[The]{Definition}
\newtheorem{Rem}[The]{Remark}
\newtheorem{Examp}[The]{Example}

\begin{document}
\title{Complete lift of vector fields and sprays  to $T^\infty M$}
\author{A. Suri }
\address{Department of Mathematics \\
Faculty of Sciences, Bu-Ali Sina University, Hamedan 65178, Iran}
\email{a.suri@basu.ac.ir \& a.suri@math.iut.ac.ir \& ali.suri@gmail.com}
\author{S. Rastegarzadeh}
\address{Department of Mathematics \\
Faculty of Sciences, Bu-Ali Sina University, Hamedan 65178, Iran}
\email{rastegarzade.math@yahoo.com}
\date{2015}
\maketitle

\begin{abstract}
In this paper for a given Banach, possibly infinite dimensional, manifold $M$ we focus on the geometry of its
iterated tangent bundle $T^rM$, $r\in {\N}\cup\{\infty\}$. First we endow $T^rM$ with a canonical atlas using that of $M$.
Then the concepts of vertical and complete lifts for functions and vector fields on $T^rM$ are defined which they will
play a pivotal role in our next studies i.e. complete lift of (semi)sprays.
Afterward we supply $T^\infty M$ with a generalized Fr\'{e}chet manifold structure and we will show that any vector field or (semi)spray  on $M$, can be lifted to a vector field or (semi)spray  on $T^\infty M$.  Then, despite of the natural difficulties with non-Banach modeled manifolds, we will discuss about the ordinary differential equations on $T^\infty M$ including  integral curves, flows  and geodesics.
Finally,  as an example, we apply our results to the infinite dimensional case of  manifold of closed curves.
\end{abstract}
\pagestyle{headings} \markright{Complete lift of vector fields and sprays  to $T^\infty M$}

\tableofcontents
\section*{Introduction}

Lift of the geometric objects to tangent bundles
had witnessed a wide interest due to the works of Miron \cite{Miron}, Bucataru and Dahl \cite{Buc-Dahl 1}, Morimoto \cite{Morimoto}, Yano, Kobayashi and Ishihara \cite{Yano, Yano-Kobayashi} and Suri \cite{Suri Osck, iso Osck}.

On the other hand, one of the most important generalizations of ordinary differential equations
to manifolds is the class of   second order differential equations (SODEs) or (semi)sprays. A semispray on $M$ is a vector filed $S$ on $TM$ such that, if $\pi:TM\to M$ denotes the canonical projection of $TM$ on $M$, then  $\pi_*\o S=id_M$ and the homogeneity requirement of order 2, leads us to the concept of sprays \cite{Buc-Dahl 1, Lang}. In fact semisprays and sprays  provide a geometric structure to study the curves on the manifold $M$ which solve special systems of SODEs with remarkable geometric meanings
\cite{Buc-Dahl 1, Del-Par, Lang}.  Moreover, the setting of semisprays provides a unified framework for studying geodesics in Riemann, Finsler and Lagrange geometries (See \cite{Lang}, \cite{Buc-Dahl} and the references therein.)

However, most of  these materials are proved for the finite dimensional case and there are a few researcher which they follow the formalism of Abraham etc. \cite{Abraham}, Lang \cite{Lang},  Hamilton  \cite{Hamilton} or Kriegel and Michor \cite{KriMic}. In this paper, for a Banach modeled manifold $M$,  first we endow the iterated tangent bundle $T^rM$, $r\in\N$, with a canonical atlas using that of $M$.
Then following the direction of \cite{Buc-Dahl 1}, the concepts of vertical and complete lifts for functions and vector fields on $T^rM$ are defined which they will
play a pivotal role in our next studies i.e. complete lifts of semisprays and sprays.

Afterward, using the algebraic tool "\textbf{projective limit}", which is compatible with our geometric setting, we will try to lift vector fields, semisprays and sprays to $T^\infty M$. More precisely we show that the iterated lifts of vector field (or semispray) form a projective system and their limits satisfy the appropriate conditions for vector fields (or semisprays) on the   Fr\'{e}chet manifold $T^\infty M$. Despite of a lack of general solvability and uniqueness theorem for ordinary differential equations for non-Banach modeled Fr\'{e}chet manifolds \cite{Hamilton}, we will prove an existence theorem for the flow of lifted vector fields on $T^\infty M$. As a consequence we show that for the given initial values, there exists a unique geodesic with respect to the lifted (semi)sprays  on $T^\infty M$ with non-trivial lifetime.
%
%
%
%
%
%
%
%
%

However, our attempt to define the  lift of functions (and forms) from the base manifold $M$ to $T^\infty M$ fails. In fact the sequence of iterated complete (and vertical) lifts of functions, does not satisfy the appropriate conditions for  projective limits of map (see remark \ref{Rem complete and vertical lifts of functions is not a projective system}).

Apart from the manifold structure of $T^\infty M$, we can define complete and vertical lifts of functions, vector fields, semisprays and sprays form $T^\infty M$ to $TT^\infty M$ similar to that in section \ref{Section Complete and vertical lifts}.  One of the  main goals of this article is to build a class of  concrete example in Fr\'{e}chet geometry for which we can find geodesics (and flows) with nontrivial lifetime (see also remark  \ref{Rem meaning of geodesics in frechet geometry}).

%
%
%
%
%
%
Finally, as an advantage of working with infinite dimensional Banach manifolds, we will apply our results to the Banach manifolds of closed curves.

Through this paper all the maps and manifolds are assumed to be smooth and the base manifold $M$ is  partitionable.

\section{Preliminaries}\label{sec1}

Let $M$ be a smooth manifold modeled on the Banach space $\E$ and
$\pi_0:TM\longrightarrow M$ be its tangent bundle. We remind that
$TM=\bigcup_{x\in M}T_xM$ such that $T_xM$ consists of all equivalent
classes of the form ${[c,x]}$ where
\begin{equation*}
c\in C_x=\{c:(\epsilon,\epsilon)\longrightarrow M;~\epsilon>0,~c~
\textrm{is smooth and}~c(0)=x\},
\end{equation*}
under the equivalence relation
\begin{equation*}
c_1\sim_x c_2\Longleftrightarrow c_1^\prime(0)=c_2^\prime(0)
\end{equation*}
for $c_1,~c_2\in C_x$. The projection map $\pi_0:TM\to M$, maps $[c,x]$ onto $x$. If
$\mathcal{A}_0=\{(\p_{\alpha_0}:=\p_\a,U_{\alpha_0}:=U_\a);~\alpha\in I\}$ is an
atlas for $M$, then we have the  canonical  atlas
$\mathcal{A}_1=\{\big(\p_{\alpha_1}:=D\p_{\a_0},U_{\a_1}:=\pi_0^{-1}(U_\alpha)\big);~\alpha\in
I\}$ for $TM$ where
\begin{eqnarray*}
\p_{\a_1}:\pi_0^{-1}(U_{\a_0})\to &\to &
U_{\a_0}\times\E\\
{[ {c},x]}&\mto& ((\p_{\a_0}\o c)(0),(\p_{\a_0}\o c)'(0)).
\end{eqnarray*}
Inductively one can define an atlas for $T^rM:=T(T^{r-1}M)$, $r\in\N$,
by
\begin{equation*}
\mathcal{A}_r:=\{\big(\p_{\a_r}:=D\p_{\a_{r-1}},U_{\a_r}:=\pi_{r-1}^{-1}(U_{\alpha_{r-1}})\big);~\alpha\in
I\}
\end{equation*}
for which $\pi_{r-1}:T^rM\to T^{r-1}M$ is the natural projection.
The model spaces for $TM$ and $T^rM$ are $\E_1:=\E^2$ and
$\E_r:=\E^{2^r}$ respectively. We add here to the convention that $T^0M=M$.
%
%

%
%
\section{Complete and vertical lifts }\label{Section Complete and vertical lifts}
The vertical and complete lift of geometric objects has been studied by many authors (see e.g \cite{Yano}, \cite{Yano-Kobayashi}, \cite{Buc-Dahl 1}, \cite{Buc-Dahl} and the references therein).
For a comprehensive treatment in the finite dimensional case  and for references to the extensive literature on
the subject one may refer to the paper \cite{Buc-Dahl 1} by Bucataru and Dahl.

In this section we introduce the concepts of vertical and complete lifts of functions and vector fields
from $T^rM$, $r\in \N\cup\{0\}$, to $T^{r+1}M$ for the Banach manifold $M$.

Set  $\k_1:=id_{TM}$ and for $r\geq 2$ consider the canonical involution $\k_r:T^rM\to
T^rM$ which satisfies
$\partial_t\partial_s\g(t,s)=\k_r\partial_s\partial_t\g(t,s)$, for
any smooth map $\g:(-\eps,\eps)^2\to T^{r-2}M$. If we consider
$T^{r-2}M$ as a smooth manifold modeled on $\E_{r-2}$, then the
charts of $T^{r-1}M$, $T^{r}M$ and $T^{r+1}M$ take their values
in $\E_{r-1}=\E_{r-2}^2$, $\E_{r}=\E_{r-2}^4$ and $\E_{r+1}=\E_{r-2}^8$ respectively. It is easy to check that the local representation of $\k_r$ is given by
\begin{eqnarray*}
{\k_r}_\a:=\p_{\a_r}\o\k_r\o{\p^{-1}_{\a_r}}:U_{\a_{r-2}}\times\E_{r-2}^3
&\to &
U_{\a_{r-2}}\times\E_{r-2}^3\\
(x,y,X,Y)&\mto& (x,X,y,Y).
\end{eqnarray*}
Setting $D^2=DD$, the relations
\begin{eqnarray}
\k_r^2 &=& Id_{T^rM}\\
\pi_r\o D\k_r&=&\k_r\o\pi_r\\
D\pi_{r-1}&=&\pi_{r}\o\k_{r+1}\\
D^2\pi_{r-1}\o\k_{r+2}&=&\k_{r+1}\o D^2\pi_{r-1}\\
D\pi_{r-1}\o \pi_{r+1}&=&   \pi_r\o D^2\pi_{r-1},
\end{eqnarray}
$r\in\N$, as they are listed in \cite{Buc-Dahl, Buc-Dahl 1},   are valid in our framework too.\\

\subsection{Vertical and complete lifts of functions and vector fields}
There are two known lifts for functions and vector fields due to \cite{Yano} and \cite{Buc-Dahl 1}. Most of the materials of this section are modified versions of those form \cite{Buc-Dahl 1}.
\begin{Def}
Suppose that $ r\geq 0 $ and $f\in{C^{\infty}(T^{r}M)}$.
Then the \textbf{vertical lift} of $f$ is the smooth function $f^{v}\in{C^{\infty}(T^{r+1}M)}$ defined by
\begin{equation}\label{f^v}
f^{v}(\xi)=(f\circ\pi_{r}\circ\k_{r+1})(\xi),\qquad{\xi}\in{T^{r+1}M}.
\end{equation}
\end{Def}
If $r=0$, then equation (\ref{f^v}) implies that $ f^{v}=f\o\pi_0$ and for $r\geq1$, the equations (2) and (4) imply that $f^{v}=f\circ{D}\pi_{r-1}$. Moreover, the local representation of $f^v$ is
\begin{eqnarray}
\label{local f^v}
f_{\a}^{v}:=f^{v}\o\p_{\a_{r+1}}^{-1}:U_{\a_{r-1}}\times\E_{r-1}^{3}
&\to &
\R
\\
\nonumber (x,y,X,Y)&\mto &{f_{\a}(x,X)}
\end{eqnarray}
where $f_\a:=f\o\p_{\a_r}^{-1}:U_{\a_{r-1}}\times \E_{r-1}\to \R$.
\begin{Def}\label{Def complete lift of function}
For $r\geq0$ and  $f\in{C^{\infty}(T^{r}M)}$ the \textbf{complete lift} of $f$ is the smooth
function $f^c\in{C^{\infty}(T^{r+1}M)}$ defined by
\begin{equation}\label{f^c}
f^{c}(\xi)=df\circ{\k_{r+1}}(\xi),\qquad{\xi}\in{T^{r+1}M}.
\end{equation}
\end{Def}
Clearly for $r=0$, $f^c(\xi)=df(\xi)$. If $r\geq 1$, then the local representation of
$f^c$  is
\begin{eqnarray*}
f_{\a}^{c}:=f^{c}\o\p^{-1}_{\a_{r+1}}:U_{\a_{r-1}}\times\E_{r-1}^3&\to &\R\\
(x,y,X,Y)&\mto &\partial_{1}f(x,X)y+\partial_{2}f(x,X)Y
\end{eqnarray*}
where $\partial_{i}, i=1,2$, stands for the partial derivative with respect to the i'th variable.

Following the formalism of \cite{Buc-Dahl 1} we define  the concepts of vertical and complete
lifts for vector field in the Banach case.
\begin{Def}
Let $r\geq0$ and  $\mathbf{A}:T^{r}M\to{T^{r+1}M}$ be a vector field.
The \textbf{vertical lift} of $\mathbf{A}$ is the vector field $\mathbf{A}^{v}:T^{r+1}M\to{T^{r+2}M}$ defined by the following relation
\begin{equation}\label {A^v}
\mathbf{A}^{v}(\xi)=D\kappa_{r+1}\circ\partial_{s}\big(\kappa_{r+1}(\xi)+s\mathbf{A}\circ\pi_{r}\circ\kappa_{r+1}(\xi)\big)\vert_{s=0},\qquad{\xi}\in{T^{r+1}M}.
\end{equation}
\end{Def}
 $\mathbf{A}^v$ is a vector field since
\begin{eqnarray*}
\pi_{r+1}\circ{\mathbf{A}^{v}(\xi)}&=&\pi_{r+1}\circ{D}\kappa_{r+1}\circ\partial_{s}\big(\kappa_{r+1}(\xi)+s\mathbf{A}\circ\pi_{r}\circ\kappa_{r+1}(\xi)\big)\vert_{s=0}\\
&=&\kappa_{r+1}\circ\pi_{r+1}\circ\partial_{s}\big(\kappa_{r+1}(\xi)+s\mathbf{A}\circ\pi_{r}\circ\kappa_{r+1}(\xi)\big)\vert_{s=0}\\
&=&\kappa_{r+1}\circ\kappa_{r+1}(\xi)\\
&=&\xi
\end{eqnarray*}
Locally on the chart $(U_{\a_r},\p_{\a_r})$, suppose that the local representation of $\mathbf{A}$ be given by
\begin{eqnarray*}
\mathbf{A}_{\a}:=\p_{\a_{r+1}}\o{\mathbf{A}}\o\p_{\a_r}^{-1}:U_{\a_{r-1}}\times\E_{r-1}&\to
&U_{\a_{r-1}}\times\E_{r-1}^3\\
(x,y) &\mto &\big(x,y;A_\a(x,y),B_\a(x,y)\big)
\end{eqnarray*}
where $A_\a$ and $B_\a$ are smooth  $\E_{r-1}$ valued functions on $U_{\a_r}\subseteq T^rM$.
Then the vertical lift of $\mathbf{A}$ is the  vector field
${\mathbf{A}}^v\in \mathfrak{X}{(T^{r+1}M)}$ with the local representation
$\mathbf{A}_{\a}^{v}:=\p_{\a_{r+2}}\o{\mathbf{A}^v}\o\p_{\a_{r+1}}^{-1}$ where
\begin{eqnarray*}
{\mathbf{A}_{\a}^v}:U_{\a_{r-1}}\times\E^{3}_{r-1}&\to
&U_{\a_{r-1}}\times\E^{7}_{r-1}\\
(x,y,X,Y) &\mto &\big(x,y,X,Y;0,A_\a(x,X),0,B_\a(x,X)\big).
\end{eqnarray*}
Note that
$A_{\a}(x,X)=A_{\a}^{v}(x,y,X,Y)$ and $B_{\a}(x,X)=B_{\a}^v(x,y,X,Y)$.

Finally we introduce the concept of complete lift of a vector field which will play a key role in our next studies.
\begin{Def}
Let $\mathbf{A}:T^{r}M\to{T^{r+1}M}$, $r\geq0$, be a vector field. The \textbf{complete lift} of $\mathbf{A}$ is the vector field $\mathbf{A}^{c}:T^{r+1}M\to{T^{r+2}M}$ defined by
\begin{equation}\label{A^c}
 \mathbf{A}^{c}=D\kappa_{r+1}\circ\kappa_{r+2}\circ{D\mathbf{A}}\circ\kappa_{r+1}
\end{equation}
\end{Def}
Note that
\begin{eqnarray*}
\pi_{r+1}\circ{\mathbf{A}^c}(\xi)&=&\pi_{r+1}\circ{D}\kappa_{r+1}\circ\kappa_{r+2}\circ{D\mathbf{A}}\circ\kappa_{r+1}(\xi)\\
&=&\kappa_{r+1}\circ\pi_{r+1}\circ\kappa_{r+2}\circ{D\mathbf{A}}\circ\kappa_{r+1}(\xi)\\
&=&\kappa_{r+1}\circ{D}\pi_{r}\circ{D\mathbf{A}}\circ\kappa_{r+1}(\xi)\\
&=&\kappa_{r+1}\circ{D}(\pi_{r}\circ{\mathbf{A}})\circ\kappa_{r+1}(\xi)\\
&=&\kappa_{r+1}\circ\kappa_{r+1}(\xi)\\
&=&\xi
\end{eqnarray*}
which means that ${\mathbf{A}^c}$ is a  vector field (see also \cite{Buc-Dahl 1}).
Moreover with the above notations the local representation of  $\mathbf{A}^c$ is
\begin{eqnarray*}
\mathbf{A}_{\a}^{c}:U_{\a_{r-1}}\times\E^{3}_{r-1}&\to&{U}_{\a_{r-1}}\times\E^{7}_{r-1}\\
\xi=(x,y,X,Y) &\mto &\big(x,y,X,Y;A^{v}_\a(\xi),A^{c}_\a(\xi),B^{v}_\a(\xi),B^{c}_\a(\xi)\big)
\end{eqnarray*}
More precisely
\begin{eqnarray*}
  \mathbf{A}_{\a}^{c}(\xi) &=& \big(x,y,X,Y; A_\a(x,X), \pa_1 A_\a(x,X)y + \pa_2 A_\a(x,X)Y \\
  && , B_\a(x,X) , \pa_1 B_\a(x,X)y + \pa_2 B_\a(x,X)Y \big).
\end{eqnarray*}

\subsection{Semisprays and their lifts}\label{def spray}

For $r\geq 1$ a semispray on $T^{r-1}M$ is a vector field $S:T^rM\to
T^{r+1}M$ with the additional property  $\k_{r+1}\o S=S$ (or equivalently $D\pi_{r-1}\o
S=id_{T^rM}$) (\cite{Buc-Dahl 1}). Considering the atlas $\mathcal{A}_{r-1}$, we observe that
locally on the chart $\big(\p_{\a_{r-1}},U_{\a_{r-1}}\big)$ the local
representation of the vector field $S$ is $S_\a:=\p_{\a_{r+1}}\o
S\o {\p_\a}_r^{-1}$ and
\begin{eqnarray*}
S_\a:U_{\a_{r-1}}\times\E_{r-1}&\to
&U_{\a_{r-1}}\times\E_{r-1}^{3}\\
(x,y) &\mto &\big(x,y;A_\a(x,y),B_\a(x,y)\big)
\end{eqnarray*}
where $A_\a , B_\a: U_{\a_{r-1}}\times\E_{r-1}\to \E_{r-1}$ are
smooth  functions locally representing $S$. The condition
$\k_{r+1}\o S=S$ (or $D\pi_r\o S=id_{T^rM}$) yields that $A_\a(x,y)=y$. By
convention, we set $B_\a(x,y):=-2G_\a(x,y)$. Hence the local representation of
$S$ is
\begin{eqnarray*}
S_\a:U_{\a_{r-1}}\times\E_{r-1}\to U_{\a_{r-1}}\times\E_{r-1}^{3}~;~~
(x,y) \mto \big(x,y;y,-2G_\a(x,y)\big).
\end{eqnarray*}
Considering the semispray $S$ as a vector field and using the notion of complete lift for vector field, we define the complete lift   $S^c={D}\k_{r+1}\o\k_{r+2}\o{DS}\o\k_{r+1}$. Bucataru and Dahl in \cite{Buc-Dahl 1} proved that $S^c$ is a semispray. Here we state a modified version of their proof for infinite dimensional manifolds.
%
%
\begin{Pro}\label{S^c}
 $S^c$ is a semispray on $T^rM$.
\end{Pro}
\begin{proof}
It suffices to show that $\k_{r+2}\o{S^c}=S^c$. For  $\xi=(x,y,X,Y)\in T^{r+1}M$ we have
\begin{eqnarray*}
\k_{r+2}\o{S^c}(\xi)&=&\k_{r+2}\o{D}\k_{r+1}\o\k_{r+2}\o{DS}\o\k_{r+1}(\xi)\\
&=&\k_{r+2}\o{D}\k_{r+1}\o\k_{r+2}\o{DS}(x,X,y,Y)\\
&=&\k_{r+2}\o{D}\k_{r+1}\o\k_{r+2}\big(S(x,X),dS(x,X)(y,Y)\big)\\
&=&\k_{r+2}\o{D}\k_{r+1}\o\k_{r+2}\big(x,X,X,-2G_{\a}(x,X);y,Y,Y\\
& &,-2dG_{\a}(x,X)(y,Y)\big)\\
&=&\k_{r+2}\o{D}\k_{r+1}\big(x,X,y,Y;X,-2G_{\a}(x,X),Y,-2dG_{\a}(x,X)(y,Y)\big)\\
&=&\k_{r+2}\big(\k_{r+1}(x,X,y,Y),D\k_{r+1}(x,X,y,Y)(X,-2G_{\a}(x,X),Y\\
& &,-2dG_{\a}(x,X)(y,Y)\big)\\
&=&\k_{r+2}\big(x,y,X,Y;X,Y,-2G_{\a}(x,X),-2dG_{\a}(x,X)(y,Y)\big)\\
&=&\big(x,y,X,Y;X,Y,-2G_{\a}(x,X),-2dG_{\a}(x,X)(y,Y)\big)\\
&=&\big(x,y,X,Y;X,Y,-2G_{\a}^{v}(\xi),-2G_{\a}^{c}(\xi)\big)\\
&=&S^{c}(\xi)
\end{eqnarray*}
which completes the proof.
\end{proof}
%
%
\begin{Def}
A geodesic for the semispray $S$ is a smooth curve $\g:(-\eps,\eps)\to
T^{r-1}M$ such that its canonical lift $\g'$ is an integral curve
for $S$ i.e. $\g''(t)=S(\g'(t))$, $t\in(-\eps,\eps)$.
\end{Def}
Locally on a chart this last means that
\begin{eqnarray*}
\Big(\g_\a(t),\g_\a'(t),\g_\a'(t),\g_\a''(t)\Big)&=&S_\a\Big((\g_\a(t),\g_\a'(t)\Big)\\
&=&\Big(\g_\a(t),\g_\a'(t),\g_\a'(t),-2G_\a(\g_\a(t),\g_\a'(t))\Big)
\end{eqnarray*}
where $\g_\a=\p_{\a_{r-1}}\o\g$.
Therefore the curve $\g$ is a geodesic for the semispray $S$ if and only if it satisfies the following second order ordinary differential equation (SODE)
\begin{equation}\label{e.geod}
\g_\a''(t)+2G_\a(\g_\a(t),\g_\a'(t))=0~~;\a\in I
\end{equation}
which are known as  geodesic equations with respect to $S$.
\begin{Rem} For a Riemannian manifold with its canonical metric spray (\cite{Lang}), equations (\ref{e.geod}) coincide with the usual geodesic equations.
\end{Rem}
%
%

%
%
%
%
\section{Complete lift of vector fields and sprays to $T^\infty M$}
The geometry of $T^rM$, $r\in\N$, and lifting of geometric objects to this bundle was studied by many authors (See \cite{Buc-Dahl 1, Yano} and the references therein). However, no extension in the direction of $T^\infty M$ has appeared in the literatures.  The aim of this section is to lift a vector field $X\in\mathfrak{X}(TM)$,  to a vector field $X^{c_\infty}\in\mathfrak{X}(TT^\infty M)$.

To assess the benefits of this lift, we will consider the interesting cases of  semisprays and sprays. In fact we will try to lift a (semi)sprays  form $M$ to $T^\infty M$.

In order to  introduce   $T^\infty M$ we will consider it as an appropriate limit (projective or inverse limit) of the finite factors $T^iM$. To make our exposition as  self-contained as possible, we state some preliminaries about projective limits of sets, topological vector spaces,
manifolds and vector bundles from \cite{split, Gal-TM, Sch}.

Let $\I$ be a directed set. Remind that $\I$ is directed if it is an ordered set with the reflexive, transitive and anti-symmetric order "$\leq$" such that for any $i,j\in\I$ there exists $k\in\I$ with $i\leq k$ and $j\leq k$.
Let $\{S_i\}_{i\in \I}$ be a family  of nonempty sets. Moreover suppose
that for $j\geq i$ in $\I$, there exists a map $f_{ji};S_j\to S_i$, known as the connecting morphism, such that
$f_{ii}=id_{S_i}$ and $f_{ji}\o f_{kj}=f_{ki}$ for $k\geq j\geq i$ in $\I$. Then $\mathcal{S}=\{S_i,f_{ji}\}_{i,j\in \I}$
is called a projective system of sets. The projective limit of this system, defined by a universal property, always exists and is (set theoretically) isomorphic to a subset of $\prod_{i\in \I}S_i $ \cite{Sch}. More precisely the projective limit of the system $\mathcal{S}$, denoted by $\varprojlim S_i$, contains those elements  $(x_i)_{i\in I}\in \prod_{i\in \I}S_i$ for which $f_{ji}(x_j)=x_i$ for all $i,j\in \I$ with $j\geq i$.

Let   $\{\E_i\}_{i\in \I}$ be family of \textbf{topological vector spaces} and $\r_{ji}:\E_j\to \E_i$, $j\geq i$ be \textbf{continuous and linear}. Then the projective  family $\{\E_i,\r_{ji}\}_{i\in \I}$ is called a projective systems of topological vector spaces. Note that   every Fr\'{e}chet space is isomorphic to a projective limit of a countable family of Banach spaces (\cite{Sch}, p. 53).

In the sequel suppose that ${\I}=\N$ and  consider the family $\mathcal{M}=\{M_i,$ $\varphi_{ji}\}_{i,j\in\N}$ where $M_i$, $i\in\N$, is a  \textbf{manifold} modeled on the  Banach space $\E_i$ and $\varphi_{ji}:M_j\to M_i$ is a differentiable map for $j\geq i$. Moreover suppose that the following conditions hold.\\
i) The model spaces $\{\E_i,\r_{ji}\}_{i,j\in\N}$ form a projective system of topological  vector spaces.\\
ii) For any $x=(x_i)_{i\in\N}\in M=\varprojlim M_i$ there exists a projective family of charts $\{(U_i,\p_i)\}$ such that $x_i\in U_i\subseteq M_i$ and for $j\geq i$, $\r_{ji}\o\p_j=\p_i\o\varphi_{ji}$ \cite{Gal-TM}.

In this case $M:=\varprojlim M_i$ may be considered as a generalized Fr\'{e}chet manifold modeled on the Fr\'{e}chet space $\E=\varprojlim \E_i$ with the atlas $\{(\varprojlim U_i,\varprojlim\p_i)\}$.

Finally, suppose that  $(E_i,p_i,M_i)$, $i\in\mathbb{N}$, be a family of  a Banach
vector bundles   with the fibres of type $\mathbb{E}^i$ respectively. Moreover suppose that
$\{(E_i,f_{ji})\}_{i\in\mathbb{N}}$ and $\{\mathbb{E}^i,\lambda_{ji}\}_{i,j\in \mathbb{N}}$ are projective systems of manifolds and  Banach spaces respectively. Then, the system $\{(E_i,p_i,M_i), (f_{ji},\varphi_{ji},\lambda_{ji})\}_{i\in\mathbb{N}}$ is called a
\textbf{strong projective system of Banach vector bundles} on  $\{(M_i,\varphi_{ji})\}_{i\in\mathbb{N}}$ if;\\
for any ${(x_i)}_{i\in\mathbb{N}}$, there exists a projective
system of trivializations $(U_i,\tau_i)$ (here
$\tau_i:{p_i}^{-1}(U_i)\longrightarrow U_i\times \mathbb{E}_i$ are
local diffeomorphisms which are linear on fibres) of $(E_i,p_i,M_i)$ such that $x_i\in M_i$,
$U=\varprojlim U^i$ is open in $M$ and
$(\varphi_{ji}\times\lambda_{ji})\circ\tau_j=\tau_i\circ f_{ji}$
for all $i,j\in{\mathbb{N}}$ with $j\geq i$.

\begin{Rem}
Let $\{M_i, \varphi_{ji}\}$ be a projective family of Banach manifolds.
It is known   that the family $\{TM_i, q_{ji}\}_{i,j\in\N}$
also form a projective system of Banach manifolds (vector bundles) \cite{Gal-TM, split} where the connecting morphisms are given by
\begin{eqnarray*}
 q_{ji}:TM_j &\to & TM_i \\
 \left[ f,x\right] ^j &\mto & [\vp_{ji}\o f , \vp_{ji}\o x]^i
\end{eqnarray*}
and $[.,.]^i$ stands for the tangent vectors in $TM_i$.

For any element $(x_i)_{i\in\N}$ in $M:=\vlim M_i$  consider the family of  charts $\{(U^i_\a, \p^i_\a)\}_{i\in\N}$ of  $\{M_i\}_{i\in\N}$, around $(x_i)_{i\in\N}$ such that $(U_\a=\vlim U^i_\a, \p_\a=\vlim\p^i_\a)$ is a chart of $M$. Fix the atlas  $\{(U_\a=\vlim U_\a^i, \p_\a=\vlim\p_\a^i)\}_{\a\in\I}$  for $M$.

Using this atlas we can define a generalized vector bundle structure  on  $\pi:TM\to M$ \cite{Gal-TM, split}.
More precisely  $\{U_\a^i,  \t_\a^i\}_{i\in\N}$ form a projective system of trivializations for the strong projective family of  vector bundles $\{(TM_i,\pi^i, M_i)$, $(q_{ji},\varphi_{ji},\r_{ji})\}_{j\geq i}$ where
\begin{eqnarray*}
\t^i_\a : (\pi^i)^{-1}(U_\a^i) &\to & U_\a^i\dar\E^i \\
{[f_i,x_i]}^i &\mto & \big(x_i, (\vp^i_\a \o f_i)^{'}(0)\big).
\end{eqnarray*}
Regarding  the fact $\vp_{ji}\o\pi^j=\pi^i\o q_{ji}$, for  $j\geq i$, we define the limit map $\pi_M:=\varprojlim \pi_i:TM\to M$.
The trivializations of $TM$ are given by $\t_\a=\vlim \t_\a^i:\pi^{-1}(U_\a)\to U_\a\times \F$ where $\F=\vlim\E_i$ \cite{Gal-TM, split}.

\end{Rem}

In order to reduce our computations, we state the following lemma.

%
%
\begin{Lem}\label{lem projective system map}
Let $\{M_i, \vp_{ji}\}_{i,j\in \N}$ and $\{N_i, \s_{ji}\}_{i,j\in \N}$ be two projective system of manifolds with the limits $M=\vlim M_i$ and $N=\vlim N_i$. The system $\{f_i:M_i\to N_i\}$ is a projective system of maps if and only if $\s_{i+1,i}\o f_{i+1} = f_i\o \p_{i+1,i}$, where
$\s_{i+1~i}:=\s_{i+1,i}$ and $\p_{i+1~i}:=\p_{i+1,i}$, for any $i\in\N$.
\end{Lem}
\begin{proof}
Let  $f_i\o\p_{i+1,i} = \s_{i+1,i}\o f_{i+1}$ for any $i\in \N$. Then for the natural number $j\in \N$ we have
\begin{eqnarray*}
\s_{ji}\o f_j &=& (\s_{i+1,i}\o ...\o\s_{j-1,j-2}\o\s_{j,j-1})\o f_j \\
&=& (\s_{i+1,i}\o ...\o\s_{j-1,j-2})\o(\s_{j,j-1}\o f_j) \\
&=& (\s_{i+1,i}\o ...\o\s_{j-1,j-2})\o(f_{j-1}\o\p_{j,j-1}) \\
&=& \s_{i+1,i}\o ...\o f_{j-2}\o\p_{j-1,j-2}\o\p_{j,j-1} \\
&\vdots &\\
&=& f_i\o(\p_{i+1,i}\o ...\o\p_{j-1,j-2}\o\p_{j,j-1}) \\
&=& f_i\o\p_{ji}.
\end{eqnarray*}
The converse is trivial.
\end{proof}
%
%
\begin{Rem}
The family ${\{T^iM\}}_{i\in\N}$ with the connecting morphisms $\pi_{i+1,i}:=\pi_i$ form a projective system of Banach manifolds. The connecting morphisms for the model spaces are $\r_{i+1,i}:\E_{i+1}=\E^{2^{i+1}}\to\E_i=\E^{2^i}$; $(x,y)\in\E_i^2\mto x\in \E_i$, $i\in\N$, and the projective family of charts is given by $\{(U_{\a_i},\p_{\a_i})\}_{i\in\N}$ with the limit  $
(U_{\a_\infty}=\vlim U_{\a_i},\p_{\a_\infty}=\vlim\p_{\a_i})$. The limit of this family is denoted by $T^\infty M:=\vlim T^iM$ and it is called the \textbf{infinite order iterated tangent bundle} of $M$.
\end{Rem}
%
%
%
%
\begin{Pro}\label{lem X^ci is a projectivesystem}
Let  $X\in\mathfrak{X}(TM)$ and set
\begin{equation}\label{complete lift of vector field}
X^{c_1}:=X^c=D\k_{2}\o \k_{3}\o DX\o\k_{2}
\end{equation}
and, for $j\in \N$, $X^{c_j}:=(X^{c_{j-1}})^c$ with $X^{c_0}=X$. The following statements hold true.\\
i) $D^2\pi_{0}\o X^{c_1}=X\o D\pi_{0}$.\\
ii) For any $i\in\N$
\begin{equation*}
D^2\pi_{i}\o X^{c_{i+1}}=X^{c_{i}}\o D\pi_{i}.
\end{equation*}
iii) For any $i,j\in \mathbb{N}$ with $j\geq i$ we have
\begin{equation}\label{X c_i form a projective system}
D^2\pi_{j,i}\o X^{c_{j}}=X^{c_{i}}\o D\pi_{j,i},
\end{equation}
where  $\pi_{j, i}:=\pi_{i}\o \pi_{i+1}\o\dots\o \pi_{j-1}:T^jM\to T^iM$.
\end{Pro}
\begin{proof}
\textbf{i.})
Using equations   $(2)$ with $r=1$  and (\ref{complete lift of vector field}) we see that
\begin{eqnarray*}
DD\pi_{0}\o X^{c_1}&=&D(\pi_1\o\k_{2})\o X^{c_1}\\
&=& D\pi_1\o D\k_{2}\o D\k_{2}\o\k_{3}\o DX\o\k_{2}  \\
&=&D\pi_1\o\k_{3}\o DX\o\k_{2}  \\
&=&\pi_{2}\o DX\o\k_{2}=X\o\pi_1\o\k_{2}=X\o D\pi_{0}.  \\
\end{eqnarray*}

\textbf{ii.}) For the natural number $i$ we see that
\begin{eqnarray*}
D^2\pi_{i}\o X^{c_{i+1}}&=&D(\pi_{i+1}\o\k_{i+2})\o X^{c_{i+1}}\\
&=& D\pi_{i+1}\o D\k_{i+2}  \o D\k_{i+2}\o\k_{i+3}\o DX^{c_i}\o\k_{i+2}\\
&=&D \pi_{i+1}\o \k_{i+3}\o DX^{c_i}\o\k_{i+2}\\
&=&\pi_{i+2}\o DX^{c_i}\o\k_{i+2}\\
&=&X^{c_i}\o\pi_{i+1}\o\k_{i+2}\\
&=&X^{c_i}\o D\pi_{i}\\
\end{eqnarray*}
which proves the second assertion. Equivalently,  the following diagram is commutative.
\[%
\begin{array}
[c]{ccccccc}%
T(T^{i+1}M)   &
\overset{D\pi_{i}}{\longrightarrow} & T(T^{i}M)   & \cdots & T(TM)  &  \overset{D\pi_{0}}{\longrightarrow} & TM  \\
X^{c_{i+1}}\downarrow &  &
\downarrow X^{c_i}&  & \downarrow X^{c_{1}}&  &\downarrow X \\
T^2(T^{i+1}M)   & \overset{D^2\pi_{i}}{\longrightarrow} &
T^2(T^{i}M)  & \cdots & T^2(TM) &
\overset{D^2\pi_{0}}{\longrightarrow}& T^2M
\end{array}
\]

\textbf{iii.}) The last part of the proposition can be deduced from part (ii) and Lemma \ref{lem projective system map}. More precisely we have the following commutative diagram.
\[%
\begin{array}
[c]{ccc}%
T(T^{j}M)   &
\overset{D\pi_{j,i}}{\longrightarrow} & T(T^{i}M)     \\
X^{c_{j}}\downarrow &  &
\downarrow X^{c_i} \\
T^2(T^{j}M)   & \overset{D^2\pi_{j,i}}{\longrightarrow} &
T^2(T^{i}M)
\end{array}
\]
\end{proof}
%
%

As a consequence,  $\{X^{c_i}\}_{i\in\N}$ form a projective system maps from $\{T(T^{i}M)$, $D\pi_{i}\}_{i\in\N}$ to $\{T^2(T^{i}M), D^2\pi_{i}\}_{i\in\N}$.
Therefore we can define $X^{c_\infty}=\vlim X^{c_i}$ from $TT^\infty M=\vlim \{T (T^iM)$, $D\pi_i\}$ to $T^2(T^\infty M)=\vlim \{T^2 (T^iM), D^2\pi_{i}\}$ (See also \cite{Gal-TM} p. 5 or \cite{split} p. 6 ).
%
%
%
\begin{The}\label{theorem X c infty exists  and is a vector field}
$X^{c_\infty}=\vlim X^{c_i}$ exists and is a vector field on $TT^\infty M$.
\end{The}
\begin{proof}
The family of vector bundles  %
$$\{(E_i,p_i,N_i):=(T^2(T^{i}M),\pi_{i+1}, T(T^{i}M))\}_{i\in\N}$$
form a strong projective systems of vector bundles (in the sense of \cite{split} Def. 4.1).  More precisely, the fibre type of the bundle $(E_i,p_i,N_i)$, $i\in\N$, is $\E^{2^{i+1}}$ with the corresponding connecting morphism
\begin{eqnarray*}
\l_{i+1,i}:\big(\E^{2^{i-1}}\big)^4=\E^{2^{i+1}}&\to& =\big(\E^{2^{i-1}}\big)^2=\E^{2^{i}}\\
(x,y,X,Y) &\mto& (x,X).
\end{eqnarray*}
As a consequence we get, $\p_{\a_{i+1}}\o D^2\pi_{i-1}=(D\pi_{i-1}\times \lambda_{i+1,i})\o\p_{\a_{i+2}}$ where $\p_{\a_{i+k}}$, $k=1,2$, are the local charts of $T^{i+k}M$.
Then, according to proposition 4.4 of \cite{split}, $T^2T^\infty M:=\vlim T^2(T^iM)$ admits a generalized Fr\'{e}chet vector bundle structure over $TT^\infty M:=\vlim T(T^iM)$.

However, $\pi_{i+1}\o D^2\pi_i=D\pi_i\o\pi_{i+2}$ for any $i\in\N$ (\cite{Buc-Dahl}, p. 2123).
\[%
\begin{array}
[c]{ccccccc}%
T(T^{i+1}M)   &
\overset{D\pi_{i}}{\longrightarrow} & T(T^{i}M)   & \cdots & T(TM)  &  \overset{D\pi_{0}}{\longrightarrow} & TM  \\
\pi_{i+2} \uparrow &  &
\uparrow \pi_{i+1}&  & \uparrow \pi_2&  &\uparrow \pi_1 \\
T^2(T^{i+1}M)   & \overset{D^2\pi_{i}}{\longrightarrow} &
T^2(T^{i}M)  & \cdots & T^2(TM) &
\overset{D^2\pi_{0}}{\longrightarrow}& T^2M
\end{array}
\]
This last means that the limit map $\vlim \pi_{i+1}$ exists and locally maps $(x,y,X$, $Y)\in(T^2\p_{\a_\infty})^{-1}(U_{\a_\infty})$ onto $(x,y)$. As a consequence $\vlim \pi_{i+1}$ exists  and
$\vlim \pi_{i+1}:T^2(T^\infty M)\to T(T^\infty M)$ is  equal to $\pi_{TT^\infty M}$. Moreover the resulting vector bundle is  a generalized vector bundle isomorphic to $\pi_{TT^\infty M}:T^2(T^\infty M)\to T(T^\infty M)$ (the tangent bundle of $T(T^\infty M)$).

According to \cite{Buc-Dahl 1}, if $X^{c_{i-1}}\in\mathfrak{X}(T(T^{i-1}M))$ then $X^{c_i}\in\mathfrak{X}(T(T^{i}M))$.
Consequently, for any $\xi\in T(T^\infty M)$, we have
\begin{equation*}
  \pi_{TT^\infty M}\o X^{c_\infty}(\xi) =\vlim\pi_{i+1}\o X^{c_\infty}(\xi)= \big(\pi_{i+1}\o X^{c_i}(\xi_i)\big)_{i\in\N} = (\xi_i)_{i\in\N}
\end{equation*}
which means that $X^{c_\infty}$ is  a vector field on   $TT^\infty M$.
\end{proof}
%

%
%
\subsection{Flow of $X^{c_\infty}$}
In closing this section, despite of the lack of a general solvability and uniqueness theorem for ordinary differential equations on non-Banach Fr\'{e}chet manifolds (\cite{Hamilton}), we prove an existence theorem for the flow of $X^{c_\infty}$.
%
%
%
%

Let $i\in\N$,  $X\in \mathfrak{X}(T^iM)$ and  $F:\mathfrak{D}(X)\subseteq T^iM\dar \R\to T^iM$ be its  flow (see e.g. \cite{Lang}). We remark that for $\xi \in T^iM$, $F(\xi,.) : I(\xi)\subseteq \R\to T^iM$; $t\mto F_t(\xi):=F(\xi,t)$, is  an  integral curve of $X$ with $F_\xi (0)=\xi$ and $I(\xi)$  is its  maximal domain (lifetime) in $\R$.

Suppose that $F:\mathfrak{D}(X)\to T^iM$ and $F^c:\mathfrak{D}(X^c)\to T^{i+1}M$ are the flows of $X$ and $X^c$ respectively. According to \cite{Buc-Dahl 1} theorem 3.6, we have $(D\pi_{i-1}\times id_\R)\mathfrak{D}(X^c)=\mathfrak{D}(X)$ and for any $(\xi,t)\in\mathfrak{D}(X^c)$,
$F_t^c(\xi)=\k_{i+1}\o D F_t\o \k_{i+1}(\xi)$ where $DF_t$ is the differential of the map $F_t:\xi\mto F_t(\xi)$. Note that the model spaces of  manifolds in \cite{Buc-Dahl 1} are finite dimensional Euclidean spaces. However, the proof of theorem 3.6. is valid for the Banach modeled manifolds as well.

%
%
%
Let $X\in\mathfrak{X}(TM)$ and $F:\mathfrak{D}(X)\subseteq TM\times \R\to TM$ be its flow. For $i\in \N$, set $X^{c_i}=(X^{c_{i-1}})^c$ and $F^{c_i}=(F^{c_{i-1}})^c$.
\begin{The}\label{theorem projective system of flow}
For $i\in \N$, suppose that $F^{c_i}:\mathfrak{D}(X^{c_i})\subseteq T(T^{i}M)\times \R\to T(T^{i}M)$ be the flow of $X^{c_i}$. Then  the following statements hold true.\\
\textbf{i.)} $\{\mathfrak{D}(X^{c_i})\subseteq T(T^{i}M)\dar \R,(D\pi_{i}\dar id_\R)|_{\mathfrak{D}(X^{c_i})}\}_{i\in\N}$ form a projective system of open sets.\\
\textbf{ii.)} $ F^{c_\infty} = \vlim F^{c_i}$ exists. \\
\textbf{iii.)} $F^{c_\infty}$ is the flow of $X^{c_\infty} $.
\end{The}
\begin{proof}

\textbf{i.)} Since $\{T^iM,\pi_{ji}\}$ is  a projective system, then $\{T(T^iM),D\pi_{ji}\}$ is  a projective system with the limit isomorphic to $ T(T^\infty M)$. It is easy to check that $\{\mathfrak{D}(X^{c_i})\subseteq T^(T^iM)\times\R,(D\pi_{ji}\times id_\R)|_{\mathfrak{D}(X^{c_i})}\}$ also form a projective system. Set   $\mathfrak{D}(X^{c_\infty}):=\vlim {\mathfrak{D}(X^{c_i})}$.  $\mathfrak{D}(X^{c_\infty})$ is an open subset of $T(T^\infty M)\times \R$ with respect to the projective topology of $T(T^\infty M)$.

\textbf{ii.)}  For any $i\in \N$ and any $\xi=(x,y,X,Y)\in T(T^{i+1}M)$ we have
\begin{eqnarray*}
D\pi_{i} \o F^{c_{i+1}}(t,\xi) &=&  D\pi_{i} \o F_t^{c_{i+1}}(x,y,X,Y)\\
&=& D\pi_{i} \o \k_{i+2} \o DF_t^{c_i} \o \k_{i+2}(x,y,X,Y)\\
&=& (\pi_{i+1}\o\k_{i+2})\o\k_{i+2} \o DF_t^{c_i}(x,X,y,Y) \\
&=& \pi_{i+1}\big(F_t^{c_i}(x,X) , dF_t^{c_i}(x,X)(y,Y) \big) \\
&=& F_t^{c_i}(x,X)
\end{eqnarray*}
and
\begin{equation*}
F^{c_i} \o (  D\pi_{i}\times id_\R)(\xi,t) = F_t^{c_i} \o D\pi_{i}(x,y,X,Y) = F_t^{c_i}(x,X)
\end{equation*}
that is the following diagram is commutative.
\[%
\begin{array}
[c]{ccc}%
\mathfrak{D}(X^{c_{i+1}})   &
\overset{F^{c_{i+1}}}{\longrightarrow} & T(T^{i+1}M)     \\
D\pi_i\times id_\R \downarrow &  &
\downarrow D\pi_i \\
\mathfrak{D}(X^{c_{i}})   & \overset{F^{c_i}}{\longrightarrow} &
T(T^{i}M)
\end{array}
\]
As a consequence of lemma \ref{lem projective system map}, $F^{c_\infty }:=\vlim F^{c_i } :\vlim \mathfrak{D}(X^{c_i })\to T(T^\infty M)$ can be defined.

\textbf{iii.)} Finally, we claim that $F^{c_\infty}:\mathfrak{D}(X^{c_\infty}):=\vlim \mathfrak{D}(X^{c_i })\to T(T^\infty M)$ is the flow of $X^{c_\infty}$. For the latter we note that $F^{c_\infty}_\xi : I(\xi)\subseteq\R\to TT^\infty M$ is an integral curve for $X^{c_\infty}$, for all $(\xi,t)\in\mathfrak{D}(X^{c_\infty})$, since
$$\frac{d}{dt} {F^{c_\infty}_\xi (t)} = \vlim\frac{d}{dt}{F_{\xi_i}^{c_i}(t)} =\vlim X^{c_i}\big({F_{\xi_i}^{c_i}}(t)\big) = X^{c_\infty}\big({F_\xi^{c_\infty}(t)}\big).$$
\end{proof}
%
%
It may be worth reminding the reader that a vector field is called complete if, all of its integral curves admit the lifetime $(-\infty , +\infty)$.
%
%
\begin{Cor}
If $X$ is  a complete vector field on $M$, then $X^{c_\infty}$ is a complete vector field on $T(T^\infty M)$ too.
\end{Cor}
%
%
With the notations as in theorem \ref{theorem projective system of flow}, we have the following useful result.
\begin{Cor}\label{cor existence of integral curve for X c infty}
For any $\xi =(\xi_i)_{i\in\N}\in T(T^\infty M)$, the unique integral curve $F^{c_\infty}_{\xi}=\vlim F^{c_i}_{\xi_i} : I(\xi)\to T(T^\infty M)$ with $F^{c_\infty}_{\xi}(0)=\xi^\infty$ exists.
\end{Cor}
%
%
%
%
\begin{Rem}\label{Rem complete and vertical lifts of functions is not a projective system}
Let $f\in C^\infty(M)$.  Set  $f^{c_1}=f^c$ and $f^{c_i}={f^{c_{i-1}}}^c$. Then according to definition \ref{Def complete lift of function},
\begin{eqnarray*}
f^{c_i}=d(f^{c_{i-1}})\o \k_{i}.
\end{eqnarray*}
If we consider the projective systems $\{T^iM,\pi_i\}_{i\in\N}$ and $\{\R,id_\R\}_{i\in\N}$ with the limits $T^\infty M=\vlim T^iM $ and $\R=\vlim \R$ respectively, then $f^{c_i}\o\pi_i\neq id_\R\o f^{c_{i+1}}$. In fact for any $(x,y,X,Y)\in T^{i+1}M$
\begin{equation*}
f^{c_i} \o \pi_i (x,y,X,Y)= f^{c_i}(x,y)
\end{equation*}
while
\begin{eqnarray*}
id_\R \o f^{c_{i+1}} (x,y,X,Y) &=& df^{c_i} \o \k_{i+1} (x,y,X,Y)= df^{c_i}(x,X,y,Y)\\
&=& \pa_1 f^{c_i}(x,X)y + \pa_2 f^{c_i}(x,X)Y.
\end{eqnarray*}
As a consequence,  the sequence $\{f^{c_{i}}\}_{i\in\N}$  is not a projective system of maps and consequently we can not define $f^{c_\infty}$ as the limit $\vlim f^{c_i}$. Similarly one can show that, the iterated  vertical lift of functions and lift of forms, generally, do not form projective systems. We refer to \cite{Ab-Man} p. 39, \cite{Asht} p. 204 and \cite{floris} p. 549 for special classes of functions and forms on projective limits of manifolds.
\end{Rem}
%
%
%
\subsection{Lift of semisprays to $T^\infty M$} In this section, for a given semispray  on $M$ we shall try to lift it to a semispray on $T^\infty M$.
%
%
\begin{Pro}\label{proposition S^infty}
For a given semispray  $S$ on $M$ its iterated complete lift $S^{c_\infty}$ is a semispray  on $T^\infty M$.
\end{Pro}
\begin{proof}
Since $S$ is a vector field on $TM$, then  proposition \ref{lem X^ci is a projectivesystem} ensures us  that $\{S^{c_i}:T(T^{i}M)\to T^2(T^{i}M)\}$ forms a projective system of vector fields with the limit $S^{c_\infty}\in\mathfrak{X}(T(T^\infty M))$.

Proposition \ref{S^c} shows that  for any $i\in\N$, $S^{c_i}=(S^{c_{i-1}})^c$ is  a semispray on $T^iM$.

Moreover one can check that $\{D\pi_{i}\}_{i\in\N}$ form a projective system maps with the limit $D\pi_{T^\infty M}$ where $\pi_{T^\infty M} = \vlim \pi_{i} :TT^\infty M \to T^\infty M$ is the tangent bundle of $T^\infty M$. Consequently for any $\xi\in T^\infty M$,
$$ D\pi_{\infty} \o S^{c_\infty}(\xi) = \vlim\big(D\pi_{i}\o S^{c_i}(\xi_i)\big) =\vlim(\xi_i)=\xi,$$
which completes the proof.
\end{proof}
%
%
%
%
%
%
%
Let $S$ be a semispray on $T^\infty M$. Keeping the formalism of section \ref{def spray}, a geodesic with respect to   $S$ is a curve $\g:(-\eps, \eps)\subset \R \to T^\infty M$ such that $\g''(t) =S(\g'(t))$, $\eps>0$ and  $t\in (-\eps, \eps)$.
%
%
Now suppose that $S$ be a semispray on $M$ and $S^{c_\infty }$ be its lift to $T^\infty M$ given by proposition \ref{proposition S^infty}.
\begin{The}\label{The existence and uniquencess of geodesic semisprays T infty M}
For any $x=(x_i)_{i\in\N} \in T^\infty M$, $\xi=(\xi_i)_{i\in\N} \in T_xT^\infty M$, there exists a unique geodesic  $\g_\xi:(-\eps, \eps)\to T^\infty M$ for $S^{c_\infty}$ such that $\g_\xi(0)= x$, ${\g_\xi}^{\prime}(0)=\xi$. Moreover  $\g_\xi=\vlim\g_{\xi_i}$ where $\g_{\xi_i}:(-\eps, \eps)\to T(T^iM)$ is a geodesic of $S^{c_i}$ with $\g_{\xi_i}(0)=x_i$ and $\g_{\xi_i}^\prime(0)=\xi_i$, $i\in \N$.
\end{The}
\begin{proof}
According to corollary \ref{cor existence of integral curve for X c infty} there exists a unique curve $F_\xi^{c_\infty}:(-\eps,\eps)\to T(T^\infty M)$ such that $\frac{d}{dt}{F_\xi^{c_\infty}}(t)=S^{c_\infty}(F_\xi^{c_\infty}(t))$, $F_\xi^{c_\infty}(0)=\xi$ and $F_\xi=\vlim F^{c_i}_{\xi_i}$.

Set $\g_\xi=\pi_{T^\infty M}\o F_\xi :(-\eps, \eps)\to T^\infty M$. We claim that $\g_\xi$ is the desired geodesic. Since
\begin{eqnarray*}
\frac{d}{dt}\g_\xi(t)&=&\frac{d}{dt}(\pi_{T^\infty M}\o F_\xi^{c_\infty})(t)=D\pi_{T^\infty M}\o \frac{d}{dt}F_\xi^{c_\infty}(t)\\
&=&D\pi_{T^\infty M}\o S^{c_\infty}(F_\xi^{c_\infty}(t))\\
&=&F_\xi^{c_\infty}(t)
\end{eqnarray*}
and $F_\xi^{c_\infty}$ is the integral curve of $S^{c_\infty}$, we conclude that $$\frac{d^2}{dt^2}\g_\xi(t)=S^{c_\infty} (\frac{d}{dt}\g_\xi(t)).$$
Moreover $\g_\xi(0)=\pi_{T^\infty M}\o F_\xi^{c_\infty}(0)=\pi_{T^\infty M}(\xi)=x$ and
\begin{eqnarray*}
\frac{d}{dt}\g_\xi(t)|_{t=0}&=&\frac{d}{dt}(\pi_{T^\infty M}\o F^{c_\infty}_\xi)(t)|_{t=0}=D\pi_{T^\infty M}\o \frac{d}{dt} F^{c_\infty}_\xi(t)_{t=o}\\
&=&D\pi_{T^\infty M}\o S^{c_\infty}(F^{c_\infty}_\xi(t))_{t=0}\\
&=&D\pi_{T^\infty M}\o S^{c_\infty}(\xi)\\
&=&\xi.
\end{eqnarray*}
\end{proof}
%
%
\subsection{Lift of sprays to $T^\infty M$}
There is a very close relations between linear connections and a special type of semisprays known as sprays \cite{Buc-Dahl 1, Lang, Del-Par, ali-iejgeo}.
\begin{Def}\label{def 2-homo spray}
A semispray $S$ on $T^{i}M$ is called  a (2-homogeneous) spray if for any $\a\in I$, $\lambda\in\R$ and $(x,\xi)\in U_{\a_i}\times \E_i$, $G_\a(x,\lambda\xi)=\lambda^2G_\a(x,\xi)$.
\end{Def}
%
%
%
\begin{Pro}\label{proposition compltete lift of sprays}The following statements hold true.\\
i) The complete lift of a spray is a spray.\\
ii) If $S$ is  a spray on $M$, then $S^{c_\infty}$ is  a spray on $T^\infty M$.
\end{Pro}

\begin{proof}
\textbf{i.}) Let $S$ be a spray on $T^iM$. Then, $S$ is a semispray. According to proposition\ref{S^c}, $S^c$ is a semispray on $T^{i+1}M$. Now, it suffices to show that the local components of $S^c$ are $2$-homogeneous.
Locally  the semispray  $S^c$  maps $\xi=(x,y,X,Y)\in T(T^{i+1}M)$ to
\begin{eqnarray*}
\Big(x,y,X,Y;X,Y,-2G_\a^v(\xi),-2G_\a^c(\xi)\Big).
\end{eqnarray*}
However, for any $\lambda\in \R$ and $\xi\in T(T^{i+1}M)$,
\begin{equation*}
G_\a^v(x,y,\lambda X,\lambda Y)=G_\a(x,\lambda X)= \lambda^2 G_\a(x,X) = \lambda^2 G_\a^v(x,y,X,Y),
\end{equation*}
and
\begin{eqnarray*}
G_\a^c(x,y,\lambda X,\lambda Y) &=& dG_\a(x,\lambda X)(y,\lambda Y)= \frac{d}{dt}G_\a\big(x+ty, \lambda (X+tY)\big)\mid_{t=0} \\
 &=& \frac{d}{dt} \lambda^2 G_\a(x+ty,X+tY)\mid_{t=0} \\
 &=& \lambda^2 \frac{d}{dt}G_\a(x+ty, X+tY)\mid_{t=0}\\
  &=& \lambda^2 dG_\a(x,X)(y,Y) \\
  &=& \lambda^2 G_\a^c(x,y,X,Y).
\end{eqnarray*}
\textbf{ii.}) Due to proposition \ref{proposition S^infty} $S^{c_\infty}$ is a semispray on $T^\infty M$ with the local components  $G_\a^\infty:=\vlim G_\a^i$ which maps $\big((x_i)_{i\in\N},(y_i)_{i\in\N}\big)\in U_\a^\infty \times \E_\infty$ to
\begin{equation*}
\Big(x,y,y,G_\a^\infty (x,y)\Big)
\end{equation*}
Part one implies that for any $i\in \N$, $S^{c_i}$ is a spray on $T^iM$. Let $\{G_\a^i\}$ be the family of  local components of $S^{c_i}$. Then,
$$G_\a^\infty(x,\lambda y)=\vlim G_\a^i(x_i,\lambda y_i)=\lambda^2\vlim G_\a^i(x_i,y_i)=\lambda^2 G_\a^\infty(x,y)$$
for any $\lambda\in \R$ and $(x,y)\in U_\a^\infty\times \E_\infty$.
\end{proof}
%
%
%
%
%
\begin{Rem}\label{Rem meaning of geodesics in frechet geometry}
Suppose that at any step we have a metric spray $S^{c_i}$ with a compatible Riemannian metric $g_i$ (for more details see  \cite{Lang}). Then, the meaning of a geodesic for the spray  $S^{c_\infty}$ on $T^\infty M$ would be (a local) length minimizer of the \textbf{sup} and \textbf{sum} metrics, in the topological sense (see \cite{Olaf} p. 1485, def. 3.5 and section 5). More precisely at any stage we have a metric space $(T^iM,d_i)$, where $d_i$ is the metric induced by $g_i$. Now, if $\g=\vlim\g_i$ is a geodesic of $S^{c_\infty}$, then at every step we have a minimizer $\g_i$ of $d_i$. As a consequence $\g$ is a minimizer of the {sup} and {sum} metrics that is, the  length minimizers on $T^\infty M$ are geodesics of $S^{c_\infty}$
\end{Rem}
%
%
%

%
%
%
\section{Applications and examples}
In this section we state an examples  to reveal the benefits of working with Banach manifolds and our results.
\begin{Examp}
Let $(M,g)$ be an $n$-dimensional Riemannian manifold and $I=[0,1]$. An $H^1$  loop $c:I\to \R^n$ is an absolutely continuous curve for which $\dot{c}(t)$ exists almost everywhere and $\dot{c}$ is square integrable and $c(0)=c(1)$ (\cite{Kling} p. 159).   A loop $c:I\to M$ is called of class $H^1$ if, for any chart $(U,\p)$ of $M$ the mapping $\p\o c:I'\subseteq c^{-1}(U)\to \mathbb{R}^n$ is  $H^1$.    Then  $H^1(I,M)$, formed by the loops $c:I\to M$ of class $H^1$, admits a Hilbert manifold structure modeled on the Hilbert space $\mathbb{H}:=H^1(c^*TM)$ \cite{Flaschel, Kling}.

The space of vector fields along $c$ is isomorphic with the space of sections $\G(c^*TM)$ of the pullback $c^* TM$ and it is identified with the tangent space $T_cH^1(I,M)$. The scalar product on $T_cH^1(I,M)$ is given by $<u,v>_c = \int_0^{1} g(u(t), v(t))dt + \int_0^{1} g\big(\frac{D}{dt}u(t), \frac{D}{dt}v(t)\big)dt$ for any $u,v\in T_cH^1(I,M)$. (For a detailed study about the geometry of $H^1(I,M)$ we refer to \cite{Kling}.) However,  one may consider the $C^k$ maps and construct a Banach manifold structure  space using the sup-norm instead of using $L^2$ norms \cite{Eliasson}.

Let $f:M\to N$ be a differentiable map between finite dimensional Riemannian manifolds. Then there is a natural (pointwise) differentiable map $H^1(f) : H^1(I,M) \to H^1(I,N)$ where $H^1(f)(c)(t) = f (c(t))$, $t\in[0,1]$. Moreover $TH^1(I,M)=H^1(I,TM)$ that is $H^1$ commutes with the functor $T$ and for the smooth map $g:L\to M$, $H^1(f\o g)=H^1(f)\o H^1(g)$  \cite{Flaschel, Kling}.

For any (semi)spray $S$ on $M$ we claim that  $$H^1S:TH^1(I,M)\to T^2H^1(I,M)$$ is a (semi)spray on $H^1(I,M)$ and $(H^1S)^c=H^1(S^c)$.

To this end, we  first show that $H(\k_i)$ is the involution map of $T^iH^1(I,M)$, $i\in \N$. For the case $i=1$ the problem is trivial. Suppose that $i\geq 2$ and $\g:(-\eps,\eps)^2\to T^{i-2}H^1(I, M)=H^1(I, T^{i-2} M)$ be a smooth curve and $k_i^{H^1}:T^iH^1(I, M)\to T^iH^1(I, M)$ be the involution map that is
\begin{eqnarray*}
k_i^{H^1}\partial_{s_1}\partial_{s_2}\g(s_1,s_2)=\partial_{s_2}\partial_{s_1}\g(s_1,s_2).
\end{eqnarray*}
Then, for any $t\in[0,1]$ we have
\begin{eqnarray*}
\big(H^1\k_i\partial_{s_1}\partial_{s_2}\g(s_1,s_2)\big)(t)&=&\k_i\partial_{s_1}\partial_{s_2}\g(s_1,s_2,t)\\
&=&\partial_{s_2}\partial_{s_1}\g(s_1,s_2,t)\\
&=&\partial_{s_2}\partial_{s_1}\g(s_1,s_2)(t)
\end{eqnarray*}
that is $H^1\k_i=\k_i^{H^1}$.

As a consequence $\k_2^{H^1} \o {H^1}S=({H^1}\k_2)\o ({H^1}S)={H^1}(\k_2\o S) ={H^1}S$. Moreover,
\begin{eqnarray*}
(H^1S)^c &=& D\k_2^{H^1} \o \k_3^{H^1}\o D H^1S \o \k_2^{H^1} \\
&=& D{H^1}\k_2 \o {H^1}\k_3 \o D {H^1}S \o {H^1}\k_2 \\
&=& {H^1}D\k_2 \o {H^1}\k_3 \o {H^1}DS \o {H^1}\k_2 \\
&=& {H^1}(D\k_2 \o \k_3 \o DS \o \k_2) \\
&=& {H^1}(S^c).
\end{eqnarray*}
Consequently  $({H^1}S)^{c_i}={H^1}(S^{c_i})$ for any $i\in N$. On the other hand, $T^iH^1(I$, $M)=H^1(I,T^iM)$. These last two equalities motivate us to define
$$T^\infty H^1(I,M)=\vlim T^iH^1(I,M)$$
and $H^1S^{c_\infty}:=(H^1S)^{c_\infty}$. However, we are not sure about the existence a reasonable manifold structure on $H^1(I,T^\infty M)$ since the manifold structure in the classical case (i.e. $H^1(I,N)$) deeply depends on the notion of the exponential mapping (and the Riemannian metric) on the target manifold which is not guaranteed in our case \cite{Eliasson, Hamilton}. In fact, on Fr\'{e}chet modeled manifolds, an inverse function theorem is not always available and only under certain conditions (like tameness \cite{Hamilton}) a version of the inverse function theorem (and consequently an exponential mapping) is available. Since for any $i\in\N$, $H^1(I,T^iM)=T^iH^1(I,TM)$ we can also define
\begin{equation*}
H^1(I,T^\infty M):=T^\infty H^1(I,TM).
\end{equation*}

Using the terminologies of Lang \cite{Lang} or Klingenberg \cite{Kling}, one can show that if $S$ is a spray then $H^1S$ is also an spray. Now, proposition \ref{proposition compltete lift of sprays} implies that $H^1S^c=(H^1S)^c$ is a spray on $TH^1(I,M)=H^1(I,TM)$ and $H^1S^{c_i}=(H^1S)^{c_i}$, for any $i\in\N$.

Finally, using remark 3.1 and the techniques of proposition 4.1 of \cite{ali-iejgeo}, enable us to lift a connection form the finite dimensional manifold $M$ to the  manifold $T^i H^1(I,M)$, $i\in\N\cup\{\infty\}$.
\end{Examp}
%
%
%
%

%
\bigskip

\end{document}